\documentclass[12pt]{amsart}

\RequirePackage{mathtools} 
\RequirePackage{amsopn}
\RequirePackage{amsfonts}
\RequirePackage{paralist}
\RequirePackage{amssymb}
\RequirePackage{amsthm}
\RequirePackage{mathrsfs}
\usepackage[normalem]{ulem}
\usepackage[alphabetic]{amsrefs}
\renewcommand\MR[1]{\relax} 
\usepackage{tikz}
\usetikzlibrary{cd,decorations.pathmorphing,decorations.markings}
\mathtoolsset{showonlyrefs}

\usepackage{hyperref}
\newtheorem{thm}{Theorem}[section]
\numberwithin{equation}{section}

\newtheorem{cor}[thm]{Corollary}
\newtheorem{lemma}[thm]{Lemma}
\newtheorem{prop}[thm]{Proposition}

\theoremstyle{definition}

\theoremstyle{remark}
\newtheorem{remark}[thm]{Remark}
\newtheorem{example}[thm]{Example}

\newtheorem{notation}[thm]{Notation}

\newtheorem{mycomment}[thm]{Comment}
{\end{mycomment}\endgroup}
\def\mathcs{C^{*}}
\newcommand{\cs}{\ensuremath{\mathcs}}

\DeclareMathSymbol{\rtimes}{\mathbin}{AMSb}{"6F}
\newcommand{\ib}{im\-prim\-i\-tiv\-ity bi\-mod\-u\-le}

\newcommand{\sme}{\,\mathord{\mathop{\text{--}}\nolimits_{\relax}}\,}

\def\R{\mathbf{R}}
\def\C{\mathbf{C}}
\def\T{\mathbf{T}}

\def\K{\mathcal{K}}

\newcommand\set[1]{\{\,#1\,\}}
\newcommand\sset[1]{\{#1\}}
\let\tensor=\otimes

\makeatletter
\def\labelenumi{\textnormal{(\@alph\c@enumi)}}
\def\theenumi{\@alph \c@enumi}
\def\labelenumii{\textnormal{(\@roman\c@enumii)}}
\def\theenumii{\@roman \c@enumii}
\newcount\charno
\def\alphapart#1{\charno=96
\advance\charno by#1\char\charno}

\makeatother

\def\<{\langle}
\def\>{\rangle}
\let\ipscriptstyle=\scriptscriptstyle
\def\lipsqueeze{{\mskip -3.0mu}}
\def\ripsqueeze{{\mskip -3.0mu}}
\def\ipcomma{\nobreak\mathrel{,}\nobreak}
\newbox\ipstrutbox
\setbox\ipstrutbox=\hbox{\vrule height8.5pt
width 0pt}
\def\ipstrut{\copy\ipstrutbox}
\def\lip#1<#2,#3>{\mathopen{\relax_{\ipstrut\ipscriptstyle{
#1}}\lipsqueeze
\langle} #2\ipcomma #3 \rangle}
\def\blip#1<#2,#3>{\mathopen{\relax_{\ipstrut
\ipscriptstyle{ #1}}\lipsqueeze\bigl\langle} #2\ipcomma #3 \bigr\rangle}
\def\rip#1<#2,#3>{\langle #2\ipcomma #3
\rangle_{\ripsqueeze\ipstrut\ipscriptstyle{#1}}}
\def\brip#1<#2,#3>{\bigl\langle #2\ipcomma #3
\bigr\rangle_{\ripsqueeze\ipstrut\ipscriptstyle{#1}}}
\def\angsqueeze{\mskip -6mu}
\def\smangsqueeze{\mskip -3.7mu}
\def\trip#1<#2,#3>{\langle\smangsqueeze\langle #2\ipcomma #3
\rangle\smangsqueeze\rangle_{\ripsqueeze\ipstrut\ipscriptstyle{#1}}}
\def\btrip#1<#2,#3>{\bigl\langle\angsqueeze\bigl\langle #2\ipcomma
#3
\bigr\rangle
\angsqueeze\bigr\rangle_{\ripsqueeze\ipstrut\ipscriptstyle{#1}}}
\def\tlip#1<#2,#3>{\mathopen{\relax_{\ipstrut\ipscriptstyle{
#1}}\lipsqueeze \langle\smangsqueeze\langle} #2\ipcomma #3
\rangle\smangsqueeze\rangle}
\def\btlip#1<#2,#3>{\mathopen{\relax_{\ipstrut\ipscriptstyle{
#1}}\lipsqueeze
\bigl\langle\angsqueeze\bigl\langle} #2\ipcomma #3
\bigr\rangle\angsqueeze\bigr\rangle}

\def\ip(#1|#2){(#1\mid #2)}
\def\bip(#1|#2){\bigl(#1 \mid #2\bigr)}
\def\Bip(#1|#2){\Bigl( #1 \bigm| #2 \Bigr)}
\expandafter\ifx\csname BibSpec\endcsname\relax\else
\BibSpec{collection.article}{%
    +{}  {\PrintAuthors}                {author}
    +{,} { \textit}                     {title}
    +{.} { }                            {part}
    +{:} { \textit}                     {subtitle}
    +{,} { \PrintContributions}         {contribution}
    +{,} { \PrintConference}            {conference}
    +{}  {\PrintBook}                   {book}
    +{,} { }                            {booktitle}
    +{,} { }                            {series}
    +{,} { \voltext}                    {volume}
    +{,} { }                            {publisher}
    +{,} { }                            {organization}
    +{,} { }                            {address}
    +{,} { \PrintDateB}                 {date}
    +{,} { pp.~}                        {pages}
    +{,} { }                            {status}
    +{,} { \PrintDOI}                   {doi}
    +{,} { available at \eprint}        {eprint}
    +{}  { \parenthesize}               {language}
    +{}  { \PrintTranslation}           {translation}
    +{;} { \PrintReprint}               {reprint}
    +{.} { }                            {note}
    +{.} {}                             {transition}
}
\BibSpec{article}{%
    +{}  {\PrintAuthors}                {author}
    +{,} { \textit}                     {title}
    +{.} { }                            {part}
    +{:} { \textit}                     {subtitle}
    +{,} { \PrintContributions}         {contribution}
    +{.} { \PrintPartials}              {partial}
    +{,} { }                            {journal}
    +{}  { \textbf}                     {volume}
    +{}  { \PrintDatePV}                {date}
    +{,} { \eprintpages}                {pages}
    +{,} { }                            {status}
    +{,} { \PrintDOI}                   {doi}
    +{,} { available at \eprint}        {eprint}
    +{}  { \parenthesize}               {language}
    +{}  { \PrintTranslation}           {translation}
    +{;} { \PrintReprint}               {reprint}
    +{.} { }                            {note}
    +{.} {}                             {transition}
}
\BibSpec{book}{%
    +{}  {\PrintPrimary}                {transition}
    +{,} { \textit}                     {title}
    +{.} { }                            {part}
    +{:} { \textit}                     {subtitle}
    +{,} { \PrintEdition}               {edition}
    +{}  { \PrintEditorsB}              {editor}
    +{,} { \PrintTranslatorsC}          {translator}
    +{,} { \PrintContributions}         {contribution}
    +{,} { }                            {series}
    +{,} { \voltext}                    {volume}
    +{,} { }                            {publisher}
    +{,} { }                            {organization}
    +{,} { }                            {address}
    +{,} { pp.~}                        {pages}
    +{,} { \PrintDateB}                 {date}
    +{,} { }                            {status}
    +{}  { \parenthesize}               {language}
    +{}  { \PrintTranslation}           {translation}
    +{;} { \PrintReprint}               {reprint}
    +{.} { }                            {note}
    +{.} {}                             {transition}
}
\fi
\evensidemargin=\oddsidemargin
\newcommand\go{G^{(0)}}

\newcommand\X{\mathsf{X}}

\newcommand\Ind{\operatorname{Ind}}

\renewcommand\H{\mathcal{H}}
\newcommand\atensor{\odot}

\renewcommand\L{\mathcal{L}}

\def\ipu(#1|#2){\bip({#1}|{#2})_{u}}

\newcommand\B{\mathcal{B}}
\newcommand\E{\mathcal{E}}

\newcommand\csgb{\cs(G;\B)}
\newcommand\csrgb{\cs_{r}(G;\B)}
\newcommand\gcgb{\Gamma_{c}(G;\B)}
\newcommand\gcgub{\Gamma_{c}(G_{u};\B)}

\def\e(#1){\mathfrak{e}(#1)}
\def\b(#1){\mathfrak{b}(#1)}

\newcommand\Ripu{\rip A_{u}}
\newcommand\Xu{\X_{u}}

\newcommand\Y{\mathsf{Y}}

\renewcommand\vec[1]{\mathbf{#1}}

\newcommand\W{\mathsf{W}}

\newcommand\dual[1]{\flat(#1)}
\newcommand\dX{\widetilde{\X}}
\newcommand\Ku{\K(\Xu)}
\newcommand\ko{\K_{0}(\Xu)}
\usepackage{bbold}
\def\charfcn#1{\mathbb{1}_{#1}}

\newcommand\ju{j_{u}}
\newcommand\Wg{\W_{\gamma}}
\newcommand\Wgo{\W_{0}^{\gamma}}

\newcommand\A{\mathcal{A}}

\newcommand\indpiu{\Ind\pi_{u}}

\definecolor{refkey}{cmyk}{0.93,0.33,0.92,0.25}
\definecolor{labelkey}{cmyk}{0.93,0.33,0.92,0.25}

\usepackage[colorinlistoftodos,textsize=tiny,textwidth=1in]{todonotes}
\allowdisplaybreaks
\begin{document}

\begin{abstract}
  If $p \colon \B\to G$ is a Fell bundle over an \'etale groupoid, then we
  show that there is an norm reducing injective linear map
  $j \colon \csrgb\to \Gamma_{0}(G;\B)$ generalizing the well know map
  $j \colon \cs_{r}(G)\to C_{0}(G)$ in the
  case of an \'etale groupoid.
\end{abstract}

\title{Renault's $j$-map for Fell bundle \cs-algebras}

\author[Duwenig]{Anna Duwenig}
\address{School of Mathematics and Applied Statistics \\ University of
  Wollongong, Wollongong, NSW 2522, Australia}
\email{aduwenig@uow.edu.au}

\author[Williams]{Dana P. Williams}
\address{Department of Mathematics\\ Dartmouth College \\ Hanover, NH
  03755-3551 USA}
\email{dana.williams@Dartmouth.edu}

\author[Zimmerman]{Joel Zimmerman}
\address{School of Mathematics and Applied Statistics \\ University of
  Wollongong, Wollongong, NSW 2522, Australia}
\email{joelz@uowmail.edu.au}

\thanks{This research was supported by the Edward Shapiro fund at
  Dartmouth College.  The first-named author was partially supported
  by a RITA Investigator grant \mbox{(IV017)}.}

\date{2 March 2022}

\maketitle


\section{Introduction}
\label{sec:introduction}

One of the key tools in working with \'etale groupoids is Renault's
``j-map'' from \cite{ren:groupoid}*{Proposition~II.4.2}.
Specifically, Renault shows that there is an injective norm reducing
linear map $j \colon \cs_{r}(G)\to C_{0}(G)$ such that
$j(f)(\gamma)=f(\gamma)$ when $f\in C_{c}(G)$.  Crucially, $j$ also
preserves the algebraic operations so that if $a,b\in\cs_{r}(G)$, then
\begin{equation}
  \label{eq:37}
  j(a^{*})(\gamma)=\overline{j(a)(\gamma^{-1})} \quad \text{and} \quad
  j(a*b)(\gamma)=\sum_{\eta\in G^{r(\gamma)}} j(a)(\eta) \,  j(b)(\eta^{-1}\gamma).
\end{equation}
In fact, Renault works with a continuous $2$-cocycle, and has observed
in \cite{ren:irms08} that, more generally, a similar result holds for
twists over \'etale groupoids.  A proof is supplied in
\cite{bfpr:nyjm21}*{Proposition~2.8}.

The purpose of this note is to extend this result to the \cs-algebra
of a Fell bundle $p \colon \B\to G$ over an \'etale groupoid $G$.
Specifically, we prove the following.

\begin{thm}
  \label{thm-main} Suppose that $p \colon \B\to G$ is a separable,
  saturated Fell bundle over a second countable, locally compact
  Hausdorff \'etale groupoid $G$.  Then there is an \emph{injective}
  norm reducing linear map $j \colon \csrgb\to \Gamma_{0}(G;\B)$ such
  that $j(f)(\gamma)=f(\gamma)$ if $f\in \gcgb$.  Furthermore, for all
  $a,b\in\csrgb$ and $\gamma\in G$, we have
  \begin{equation}
    \label{eq:43}
    j(a^{*})(\gamma)=j(a)(\gamma^{-1})^{*}
  \end{equation}
  and
  \begin{equation}
    \label{eq:44}
    j(a*b)(\gamma)=\sum_{\eta\in G^{r(\gamma)}} j(a)(\eta)
    \, j(b)(\eta^{-1}\gamma)
  \end{equation}
  where the sum converges in the norm topology in $B_{\gamma}$.
\end{thm}

The existence of $j$ is reasonably straightforward as is the adjoint
property \eqref{eq:43}.  However, to be truly useful in applications,
we clearly need to establish the convolution result \eqref{eq:44}.
Unlike the situation of Renault's original result, the convolution
result requires considerable technology.  Renault is able to show that
the sum in \eqref{eq:37} converges absolutely basically using
H\"older's inequality.  Here we have to use the internal tensor
product of Hilbert modules.  As a consequence, we are only able to
show convergence of the sum in \eqref{eq:44} in the strong sense
described in Remark~\ref{rem-sums}.  We suspect that the sum does not
converge absolutely in general.

Our result subsumes earlier results for crossed
  products by discrete groups twisted by a unitary $2$-cocycle in
  \cite{zel:jmpa68}.

\subsubsection*{Assumptions}
\label{sec:assumptions}

Whenever possible we assume that our topological spaces are second
countable and that our \cs-algebras are separable.  In particular, $G$
will always denote a second countable, locally compact Hausdorff
\'etale groupoid.  We also adopt the standard convention that
homomorphisms between \cs-algebras are $*$-preserving.

\section{Preliminaries}
\label{sec:preliminaries}

Fell bundles over groupoids are a natural generalization of Fell's
\cs-algebraic bundles from \cite{fd:representations2}*{Chapter~VIII}.
They were introduced in \cite{yam:xx87}.  For more details, see
\cites{kum:pams98, muhwil:dm08}.  Roughly speaking, a Fell bundle $\B$
over a locally compact Hausdorff groupoid $G$ is a
(upper-semicontinuous) Banach bundle $p \colon \B \to G$ endowed with
a continuous involution $b \mapsto b^*$ and a continuous
multiplication $(a,b) \mapsto ab$ from
$\B^{(2)}=\set{(b,b'):(p(b),p(b')\in G^{(2)}}$ to $\B$ such
that---with respect to the operations, actions, and inner products
induced by the involution and multiplication---the fibres
$B_{u} = p^{-1}(u)$ over units $u \in \go$ are $C^*$-algebras and such
that each fibre $B_{\gamma}=p^{-1}(\gamma)$ is a
$B_{r(\gamma)}$--$B_{s(\gamma)}$-imprimitivity bimodule.

Our references for Banach bundles are \cite{muhwil:dm08}*{Appendix~A}
and, for the \cs-bundle case, \cite{wil:crossed}*{Appendix~C}.  As is
now common practice, we drop the adjective ``upper-semicontinuous'' in
front of the term ``Banach bundle''.  If the map $a\mapsto \|a\|$ is
continuous rather than merely upper-semicontinuous, then we include
the adjective continuous.  An excellent reference for continuous
Banach bundles is \S\S13--14 of \cite{fd:representations1}*{Chap. II}.
If $p \colon \B\to X$ is a Banach bundle, we will write
$\Gamma_{0}(X;\B)$ for the continuous sections of $\B$ which vanish at
infinity.  Furthermore, $\Gamma_{0}(X; \B )$ is a Banach space with
respect to the supremum norm (see \cite{dg:banach}*{p.~10} or
\cite{wil:crossed}*{Proposition~C.23}).

We write $\Gamma_c(G; \B)$ for the $*$-algebra of continuous compactly
supported sections of $\B$ under convolution and involution.  If
$Y\subset G$, then we write $\Gamma_{c}(Y;\B)$ for the continuous
compactly supported sections of $\B$ restricted to $Y$.  Of course
such sections take values in $\B_{Y} \coloneqq p^{-1}(Y)$ which is
usually called the the restriction of $\B$ to $Y$. If $Y$ is closed,
then any $f\in \Gamma_{c}(Y;\B_{Y})$ is the restriction of a section
in $\Gamma_{c}(G;\B)$ to $Y$ by the Tietze Extension Theorem for
Banach bundles \cite{muhwil:dm08}*{Proposition~A.5}.

As in \cite{simwil:nyjm13}, we note that $A=\Gamma_{0}(\go;\B)$ is a
\cs-algebra called the \cs-algebra of $\B$.  If $u\in \go$, then we
often write $A_{u}$ for $B_{u}$ when we want to think of $B_{u}$ as a
\cs-algebra.  Then, for example, $B_{\gamma}$ is a
$A_{r(\gamma)}\sme A_{s(\gamma)}$-\ib.  We rely on
\cite{simwil:nyjm13}*{\S4} for the definition of, and standard results
for, the reduced norm on $\gcgb$. In keeping with our standing
assumptions, we assume that our Banach bundles $p \colon \B\to G$, and
hence our Fell bundles, are separable in the sense that
$\Gamma_{0}(G;\B)$ is a separable Banach space.

  \begin{remark}
    It was observed in \cite{bmz:pems13}*{Lemma~3.16} that any Fell
    bundle over a \emph{group} is necessarily a \emph{continuous}
    Banach bundle.  This is also the case for a Fell bundle $\B$ over
    a groupoid $G$ in the case where the associated \cs-algebra
    $A=\Gamma_{0}(\go;\B)$ is a {\em continuous} \cs-bundle over
    $\go$.
  \end{remark}

  Note that the topology on the total space $\B$ is not necessarily
  Hausdorff when $p \colon \B\to G$ is not a continuous Banach bundle
  (see \cite{wil:crossed}*{Example~C.27}).  However, the following
  observation often allows us to finesse this difficulty.  We include
  a proof for convenience.\footnote{For continuous bundles, this is
    \cite{fd:representations1}*{Proposition~II.13.11}.}

\begin{lemma}\label{lem-phew}
  Let $p \colon \B\to X$ be a Banach bundle over a locally compact
  Hausdorff space $X$.  Then the relative topology on each fibre
  $B_{x}$ is the norm topology.
\end{lemma}
\begin{proof}
  Suppose $(a_{i})$ is a net in $B_{x}$ converging to $a$ in $\B$.
  Since $X$ is Hausdorff and $p$ continuous, $a\in B_{x}$, and
  $(a_{i}-a)$ converges to $0_{x}$ in $\B$.  Since $b\mapsto \|b\|$ is
  upper-semicontinuous, $\set{b:\|b\|<\epsilon}$ is a neighborhood of
  $0_{x}$ in $\B$.  It follows that $\|a_{i}-a\|\to 0$.

  If $(a_{i})\to a$ in norm in $B_{x}$, then $\|a_{i}-a\|\to 0$.  It
  now follows from the Banach bundle axioms (for example axiom~B4 from
  \cite{muhwil:dm08}*{Definition~A.1}), that $a_{i}-a\to 0_{x}$ in
  $\B$.  Then $a_{i}\to a$ in $\B$.
\end{proof}

\begin{remark}
  [Sums] \label{rem-sums} We should explain what we mean by sums such
  as \eqref{eq:44} when $G^{r(\gamma)}$ is infinite.  If
  $f \colon X\to V$ is a function on a set $X$ taking values in a
  topological vector space $V$, then
  \begin{equation}
    \label{eq:62}
    \sum_{x\in X}f(x)
  \end{equation}
  is defined to be the limit, if it exists, of the net $(s_{F})$ where
  $F$ is finite subset of $X$, $s_{F}=\sum_{x\in F}f(x)$, and
  $(s_{F})$ is directed by containment: $F_{1}\le F_{2}$ if
  $F_{1}\subset F_{2}$.  It is not hard to see that if $X$ is
  countably infinite and if $(x_{n})$ is any enumeration of $X$, then
  \begin{equation}
    \label{eq:63}
    \sum_{n=1}^{\infty}f(x_{n})
  \end{equation}
  converges if \eqref{eq:62} does and then the sums coincide.  In
  particular, the sum in \eqref{eq:63} is invariant under
  rearrangement if \eqref{eq:62} converges.
\end{remark}

Note that if for some enumeration of $X$, the sum in \eqref{eq:63}
converges, it is not necessarily the case that \eqref{eq:62}
converges---consider a conditional convergent series in $\R$.  But if
$V$ is a \cs-algebra and if $f(x)$ is always positive in $V$, then the
converse holds.

Recall from \cite{rw:morita}*{pp. 49--50} that if $\X$ is an
$A\sme B$-\ib, then the dual module $\dX$ is the $B\sme A$-\ib\
defined as follows.  We let $\dX$ be the conjugate vector space to
$\X$.  This means that $\dX$ is equal to $\X$ as a set.  If
$\flat \colon \X\to\dX$ is the identity map, then
$\dual{\lambda\cdot x}=\overline\lambda\cdot \dual x$ for
$\lambda\in\C$.  The left $B$-action on $\dX$ is then given by
$b\cdot \dual x = \dual {x\cdot b^{*}}$ and the left $B$-valued inner
product is given by $\lip B<\dual x, \dual y>=\rip B<x,y>$.  Similar
formulas hold for the right $A$-action and right $A$-valued inner
product.

If $\X$ is an $A\sme B$-\ib, and if $\Y$ is a $B\sme C$-\ib, then the
internal tensor product $\X\tensor_{B}\Y$ is a $A\sme C$-\ib\ with
respect to the obvious actions and the inner product given, for
example, in \cite{rw:morita}*{Proposition~3.16}.

\begin{lemma}
  \label{lem-cross-norm} Suppose that $\X$ is a right Hilbert
  $B$-module and that $\Y$ is a $B\sme A$-\ib.  Then if
  $x\tensor y\in \X\tensor_{B}\Y$, then
  \begin{equation}
    \label{eq:60}
    \|x\tensor y\|_{\X\tensor_{B}\Y}\le \|x\|_{\X}\|y\|_{\Y}.
  \end{equation}
\end{lemma}
\begin{proof}
  We have
  \begin{align}
    \label{eq:61}
    \|x\tensor y\|^{2}_{\X\tensor_{B}\Y}
    &=\| \rip A<x\tensor y, x\tensor y>\| \\
    &= \|\brip A<{\rip B<x,x>}\cdot y,y> \| \\
    \intertext{which, by 
    Cauchy--Schwarz
    (see \cite{rw:morita}*{Lemma
    2.5 and Corollary 2.7}), is}
    &\le \|\rip B<x,x>\cdot y\|_{\Y}\|y\|_{\Y}\\
    &\le \|\rip
      B<x,x>\|\|y\|_{\Y}^{2}=\|x\|_{\X}^{2} \|y\|_{\Y}^{2}.\qedhere
  \end{align}
\end{proof}

\section{The Module $\Xu$}
\label{sec:reduced-norm}

Since we are always assuming that $G$ is \'etale, for any $u\in \go$,
$G_{u}$ is a closed, discrete subset of $G$.  Hence $\gcgub$ is just
the set of finitely supported functions $f$ on $G_{u}$ such that
$f(\eta)\in B_{\eta}$ for each $\eta\in G_{u}$. We let $\Xu$ be the
full right Hilbert $A_{u}$-module that is the completion of
$\Gamma_{c}(G_{u};\B)$ with respect to the pre-inner product
\begin{equation}
  \label{eq:1}
  \Ripu <h,k>=\sum_{\eta\in G_{u}}h(\eta)^{*} \, k(\eta),
\end{equation}
equipped with the obvious right $A_{u}$-action.  As always, denote by
$\L(\Xu)$ the \cs-algebra of adjointable operators on $\Xu$.

\begin{lemma}\label{lem:Vu}
  For each $u\in\go$, there is a homomorphism
  $V_{u} \colon \csrgb\to\L(\Xu)$ such that for all
    $\zeta\in G_{u}$
  \begin{equation}
    \label{eq:2a}
    V_{u}(f)(h)(\zeta)=\sum_{\eta\in
      G^{r(\zeta)}}f(\eta) \, h(\eta^{-1}\zeta)\quad\text{for
      $f\in\Gamma_{c}(G;\B)$ and $h\in\Gamma_{c}(G_{u};\B)$}.  
  \end{equation}
  Furthermore for all $a\in\csrgb$,
  $\|a\|_{r}=\sup_{u\in\go}\|V_{u}(a)\|_{\Xu}$ where $\|\cdot\|_{\Xu}$
  is the norm on $\Xu$.\footnote{Since $\Xu$ is Hilbert
    $A_{u}$-module, its norm, $\|\cdot\|_{\Xu}$, is often denoted by
    $\|\cdot \|_{A_{u}}$.  However, this notation would prove
    confusing in the sequel.}
\end{lemma}
\begin{proof}
  Observe that $\Xu$ is the right Hilbert $A_{u}$-module constructed
  in \cite{simwil:nyjm13}*{\S4.1} for the subgroupoid $H=\sset u$;
  that is, $\X_{u}$ is the module used to induce representations from
  $A_{u}$ to $\csgb$.  Moreover, as described in
  \cite{simwil:nyjm13}*{\S4.2}, such induced representations are
  regular representations and factor through $\csrgb$.  Furthermore,
  there is a nondegenerate homomorphism $V \colon \csgb\to \L(\Xu)$
  satisfying \eqref{eq:2a}.  Hence if $\pi_{u}$ is a faithful
  representation of $A_{u}$, then as in
  \cite{simwil:nyjm13}*{Example~13}, the regular representation
  $\indpiu$ acts on the completion of $\gcgub\atensor \H_{\pi_{u}}$
  with respect to the pre-inner product determined on elementary
  tensors by
  \begin{equation}
    \label{eq:64}
    \ip(f\tensor h| g\tensor k)=\sum_{\eta\in G_{u}}
    \bip(\pi_{u}(g(\eta)^{*}f(\eta))h|k) .
  \end{equation}
  In particular, arguing as in the proof of
  \cite{simwil:nyjm13}*{Lemma~9}, we see that $\ker V=\ker(\indpiu)$.
  Therefore $\ker V\supset \set{a\in\csgb:\|a\|_{r}=0}$ and $V$
  factors through a homomorphism $V_{u}$ as claimed.  Since
  $\bigoplus_{u\in\go}\pi_{u}$ is a faithful representation of
  $A=\Gamma_{0}(\go;\B)$, the $V_{u}$ determine the reduced norm as
  claimed.
\end{proof}

\begin{lemma}\label{lem-xu-form}
  We can realize $\X_{u}$ as the right Hilbert $A_{u}$-module
  \begin{equation}
    \label{eq:1a}
    \Bigl\{\,\vec x \in \prod_{\eta\in G_{u}}B_{\eta}:\text{$\sum_{\eta\in
        G_{u}}\vec x(\eta)^{*} \, \vec x(\eta)$ converges in $A_{u}$}\,\Bigr\}
  \end{equation}
  equipped with the inner product and right $A_{u}$-action given by
  \begin{equation}
    \label{eq:3}
    \Ripu <\vec x,\vec y>=\sum_{\eta\in G_{u}}\vec x(\eta)^{*} \, \vec
    y(\eta)\quad\text{and} \quad(\vec x\cdot a)(\eta)=\vec
    x(\eta)\cdot a.
    \quad 
  \end{equation}
\end{lemma}
\begin{remark}
  \label{rem-sum-conv} Part of the conclusion of
  Lemma~\ref{lem-xu-form} is that the sum in \eqref{eq:3} converges in
  $A_{u}$ as in Remark~\ref{rem-sums}.
\end{remark}
\begin{proof}
  Arguing as in \cite{rw:morita}*{Proposition~2.15}, or more
  generally, \cite{lan:hilbert}*{p.~6}, \eqref{eq:1a} determines a
  right Hilbert $A_{u}$-module.  Let $\|\cdot\|_{0}$ be the
  corresponding Hilbert $A_{u}$-module norm induced by the inner
  product given in \eqref{eq:3}.  If $\vec x$ belongs to \eqref{eq:1a}
  and if $F$ is a finite subset of $G_{u}$, then we let
  $\vec x_{F}=\charfcn F \vec x$. We claim that the net $(\vec x_{F})$
  converges to $\vec x$.  In fact,
  \begin{align}
    \|\vec x -\vec x_{F}\|_{0}^{2}
    &=\|\rip A_{u}<\vec x-\vec x_{F},\vec x -\vec x_{F}>\| \\
    &= \| \rip A_{u}<\vec x,\vec x> - \rip A_{u}<\vec x,\vec x_{F}> -
      \rip A_{u}<\vec x_{F},\vec x> + \rip A_{u}<\vec x_{F},\vec
      x_{F}>\| \\
    &=\|\rip A_{u}<\vec x,\vec x> -\rip A_{u}<\vec x_{F},\vec x_{F}>\|,
      \label{eq:59a}
  \end{align}
  where the last equality follows from the fact that
  \begin{equation}
    \label{eq:2}
    \rip A_{u}<\vec x,\vec x_{F}>=
    \rip A_{u}<\vec x_{F},\vec x> = \rip A_{u}<\vec x_{F},\vec
    x_{F}>.
  \end{equation}
  Clearly, \eqref{eq:59a} tends to zero with $F$.  Hence, we can view
  $\gcgub$ as a dense subspace.  The result follows from this.
\end{proof}

\begin{notation}
  \label{rem-dense}
  If $\gamma\in G_{u}$ and $b\in B_{\gamma}$, we let $h_{\gamma}^{b}$
  be the section that takes the value $b$ at $\gamma$ and is zero
  elsewhere. Since $G_{u}$ is discrete, $h_{\gamma}^b$ is an element
  of $\gcgub$, and elements of this form span a dense subspace of
  $\Xu$.  Note that
  \begin{equation}
    \label{eq:8a}
    \|h_{\gamma}^{b}\|_{\Xu}=\|b\|.
  \end{equation}
\end{notation}

\begin{prop}
  \label{prop-j-inj} There is an
  \emph{injective} 
  norm reducing linear map $j \colon \csrgb\to \Gamma_{0}(G;\B)$ such
  that $j(f)(\gamma)=f(\gamma)$ for all $f\in\gcgb$.
\end{prop}
\begin{proof}
  Since
  \begin{align}
    \label{eq:4a}
    \Ripu<h,V_{u}(f)k>
    = \sum_{\zeta\in G_{u}} h(\zeta)^{*} \, 
    \bigl[V_{u}(f)k\bigr]
    (\zeta) 
    = \sum_{\zeta\in G_{u}} \sum_{\eta\in G^{r(\zeta)}} h(\zeta)^{*} \, 
    f(\eta)  \, k(\eta^{-1}\zeta),
  \end{align}
  we have
  \begin{align}
    \label{eq:6a}
    \Ripu<h_{\gamma}^{b},V_{u}(f) h_{\beta}^{c}>
    =b^{*}f(\gamma\beta^{-1})c\quad\text{for all $\gamma,\beta\in G_{u}$
    and $f\in \gcgb$.}
  \end{align}
  In particular, if $(e_{n})$ is an approximate identity for $A_{u}$
  \begin{align}
    \label{eq:7a}
    \Ripu<h_{\gamma}^{f(\gamma)},V_{u}(f) h_{u}^{e_{n}}>
    =f(\gamma)^{*}f(\gamma)e_{n}. 
  \end{align}
  Thus, by Cauchy--Schwarz,
  \begin{align}
    \label{eq:9a}
    \|f(\gamma)^{*}f(\gamma)e_{n}\|
    &=
      \|\Ripu<h_{\gamma}^{f(\gamma)},V_{u}(f) h_{u}^{e_{n}}>\|
      \le
      \|h_{\gamma}^{f(\gamma)}\|\|V_{u}(f)\|\|h_{u}^{e_{n}}\|\\
    \intertext{which, by \eqref{eq:8a} and since $\|e_n\|\leq 1$, is}
    &
      \le
      \|f(\gamma)\|_{B_{\gamma}}\| \|V_{u}(f)\|
    \\
    &\le
      \|f(\gamma)\|_{B_{\gamma}}\|f\|_{r}.
  \end{align}
  On the other hand,
  $f(\gamma)^{*}f(\gamma)e_{n} \to f(\gamma)^{*}f(\gamma)$ in $A_{u}$.
  Therefore
  \begin{equation}
    \label{eq:10}
    \lim_{n}\|f(\gamma)^{*}f(\gamma)e_{n}\|_{A_{u}}
    =\|f(\gamma)^{*}f(\gamma)\|_{A_{u}}=\|f(\gamma)\|^{2}_{B_{\gamma}}.  
  \end{equation}
  It follows that for any $\gamma \in G_{u}$,
  \begin{equation}
    \label{eq:11}
    \|f(\gamma)\|_{B_{\gamma}}\le\|f\|_{r}.
  \end{equation}
  Since $u\in\go$ was arbitrary, we have
  \begin{equation}
    \label{eq:12}
    \|f\|_{\infty}\le \|f\|_{r}\quad\text{for any $f\in\gcgb$.}
  \end{equation}

  It follows that the map sending $f\in \gcgb\subset \csrgb$ to
  $f\in \Gamma_{0}(G;\B)$ is a bounded linear map.  Hence we can
  extend to the completion and obtain a linear map
  \begin{align}
    \label{eq:13}
    j \colon \csrgb\to \Gamma_{0}(G;\B)
  \end{align}
  such that $\|j(f)\|_{\infty}\le \|f\|_{r}$ and such that $j(f)=f$
  for all $f\in\gcgb$.

  To see that $j$ is injective, it suffices to see that if
  $a\in\csrgb$ and $a\not=0$, then $j(a)\not=0$.  But if $a\not=0$,
  then by Lemma~\ref{lem:Vu}, there is a $u\in\go$ such that
  $V_{u}(a)\not=0$.  Since elements of the form $h_{\gamma}^{b}$ with
  $\gamma\in G_{u}$ and $b\in B_{\gamma}$ span a dense subspace of
  $\X_{u}$, there must be vectors $h_{\gamma}^{b}$ and $h_{\beta}^{c}$
  in $\Gamma_{c}(G_{u};\B)\subset \X_{u}$ such that
  \begin{equation}
    \label{eq:14}
    \Ripu<h_{\gamma}^{b},V_{u}(a)h_{\beta}^{c}>\not=0.
  \end{equation}
  Let $(f_{n})$ be a sequence in $\gcgb$ converging to $a$ in
  $\csrgb$.  Since $j$ is bounded, $j(f_{n})\to j(a)$ in
  $\Gamma_{0}(G;\B)$.  In particular,
  \begin{equation}
    \label{eq:17}
    j(a)(\gamma\beta^{-1})=\lim_{n}j(f_{n})(\gamma\beta^{-1}).
  \end{equation}
  Since multiplication in $\B$ is continuous and associative, this
  implies
  \begin{equation}\label{eq:17b}
    b^{*} j(a)(\gamma\beta^{-1})c=\lim_{n} b^{*}
    j(f_{n})(\gamma\beta^{-1})c. 
  \end{equation}
  On the other hand, by \eqref{eq:6a},
  \begin{equation}
    \label{eq:15}
    b^{*}j(f_{n})(\gamma\beta^{-1})c
    =
    b^{*}f_{n}(\gamma\beta^{-1})c
    =
    \Ripu<h_{\gamma}^{b},V_{u}(f_{n}) h_{\beta}^{c}>.
  \end{equation}
  Since $V_{u}$ is bounded,
  \begin{equation}
    \label{eq:16}\lim_{n}
    \Ripu<h_{\gamma}^{b},V_{u}(f_{n}) h_{\beta}^{c}>=
    \Ripu<h_{\gamma}^{b},V_{u}(a) h_{\beta}^{c}>. 
  \end{equation}
  We have shown that $b^{*} j(f_{n})(\gamma\beta^{-1})c$ converges to
  both $b^{*} j(a)(\gamma\beta^{-1})c$ and
  $\Ripu<h_{\gamma}^{b},V_{u}(a) h_{\beta}^{c}>$ in a fibre of
  $\B$. Since those fibres are Hausdorff by Lemma~\ref{lem-phew}, we
  conclude
  \begin{equation}
    \label{eq:18}
    \Ripu<h_{\gamma}^{b},V_{u}(a)h_{\beta}^{c}>=b^{*}j(a)(\gamma\beta^{-1})c.
  \end{equation}
  
  Since the left-hand side of \eqref{eq:18} is nonzero by assumption,
  it follows that $j(a)(\gamma\beta^{-1})\not=0$.  Thus, $j$ is
  injective.
\end{proof}

\begin{cor}
  \label{cor-j-star} If $a\in \csrgb$, then
  $j(a^{*})(\eta)=j(a)(\eta^{-1})^{*}$.
\end{cor}
\begin{proof}
  Let $(f_{n})$ be a sequence in $\gcgb$ such that $f_{n}\to a$ in
  $\csrgb$.  Then for all $\eta\in G$,
  $f_{n}(\eta)=j(f_{n})(\eta)\to j(a)(\eta)$ and
  $f^{*}_{n}(\eta)=j(f_{n}^{*})(\eta)\to j(a^{*})(\eta)$.  But by
  definition, $f^{*}_{n}(\eta)=f_{n}(\eta^{-1})^{*}$.  Hence
  $(f_{n}(\eta^{-1})^{*})$ converges to both $j(a)(\eta^{-1})^{*}$ and
  $j(a^{*})(\eta)$ in $B_{\eta}$.  Since the relative topology on
  $B_{\eta}$ is the norm topology, it is Hausdorff, and
  $j(a^{*})(\eta) = j(a)(\eta^{-1})^{*}$ as required.
\end{proof}

If $a\in\csrgb$, then for any $u\in\go$, the restriction of $j(a)$ to
$G_{u}$ defines an element $\ju(a)$ of
$\prod_{\eta\in G_{u}}B_{\eta}$.  In particular, we have the
following.

\begin{prop}
  \label{prop-ja-xu} If $a\in\csrgb$ and $u\in\go$, then
  $\ju(a)\in \Xu$ and
  \begin{equation}
    \label{eq:4}
    \|\ju(a)\|_{\Xu}\le \|a\|_{r}.
  \end{equation}
\end{prop}
\begin{proof}
  Suppose that $a=f\in\gcgb$.  Then if $(e_{n})$ is an approximate
  identity for $A_{u}$, we have
  \begin{equation}
    \label{eq:6}
    \bigl[V_{u}(f)h^{e_{n}}_{u}\bigr](\gamma)=f(\gamma) e_{n}.
  \end{equation}
  Thus, on the one hand,
  \begin{equation}
    \label{eq:26}
    \|V_{u}(f) h_{u}^{e_n}\|^{2}_{\Xu} \le \|f\|_{r}^{2}\|h_{u}^{e_{n}}
    \|_{\Xu}^{2}\le 
    \|f\|_{r}^{2}. 
  \end{equation}
  On the other hand,
  \begin{align}
    \|V_{u}(f) h_{u}^{e_n}\|^{2}_{\Xu}
    &=\|\Ripu<V_{u}(f)h_{u}^{e_{n}},V_{u}(f) h_{u}^{e_n}> \|_{A_{u}}\\
    &=\Bigl\|\sum_{\eta\in G_{u}}e_{n} \, f(\eta)^{*} \, f(\eta) \,  e_{n}
      \Bigr\|_{A_{u}}=\Bigl\| e_{n}\Bigl( 
      \sum_{\eta\in G_{u}}f(\eta)^{*} \, f(\eta)\Bigr)e_{n}\Bigr\|_{A_{u}}.
      \label{eq:27}
  \end{align}
  Note that as $f$ is compactly supported, only finitely many of the
  summands in each sum in \eqref{eq:27} are nonzero, so each sum
  converges to an element of $A_{u}$.  Thus, for all $n$,
  \begin{equation}
    \label{eq:28}
    \Bigl \| e_{n}\Bigl( \sum_{\eta\in G_{u}}f(\eta)^{*} \, f(\eta)
    \Bigr)e_{n} 
    \Bigr \|_{A_{u}}\le \|f\|_{r}^{2}.
  \end{equation}
  We conclude that for all $f\in \gcgb$,
  \begin{equation}
    \label{eq:29}
    \|f\|_{\Xu}^{2}=\Bigl \| \sum_{\eta\in G_{u}}f(\eta)^{*} \, f(\eta)
    \Bigr \|_{A_{u}}\le \|f\|_{r}^{2}.
  \end{equation}
  Thus, \eqref{eq:4} holds provided $a=f\in \gcgb$.

  Suppose that $(f_{n})$ is a sequence in $\gcgb$ such that
  $f_{n}\to a$ in $\csrgb$.  Since
  \begin{equation}
    \label{eq:71}
    \|j_{u}(f_{n})-j_{u}(f_{m})\|_{\Xu}=\|j_{u}(f_{n}-f_{m})\|_{\Xu} \le
    \|f_{n}-f_{m}\|_{r} 
  \end{equation}
  by \eqref{eq:29}, it follows that $\bigl(j_{u}(f_{n})\bigr)$ is
  Cauchy in $\Xu$.  Since $\Xu$ is complete,
  $\bigl(j_{u}(f_{n})\bigr)$ converges to some $\vec y\in\Xu$.  Since
  the map $\vec x \mapsto \vec x(\eta)$ is clearly a bounded linear
  map from $\Xu$ to $B_{\eta}$, ``evaluation at $\eta$'' is continuous
  on $\Xu$.  Hence $j_{u}(f_{n})(\eta)\to \vec y(\eta)$ in $B_{\eta}$
  for all $\eta\in G_{u}$. Further, as $j$ is continuous, we have
  $j(f_{n})\to j(a)$ uniformly.  Since $B_{\eta}$ is Hausdorff, we
  conclude that $\vec y=j_{u}(a)$.  Hence, $j_{u}(a)\in\Xu$ and
  $\|j_{u}(a)\|_{\Xu}= \lim_{n}\|j_{u}(f_{n})\|_{\Xu}\le
  \lim_{n}\|f_{n}\|_{r}= \|a\|_{r}$.
\end{proof}

As a corollary of the proof, we have the following.

\begin{cor}
  \label{cor-j-xu} Suppose that $(f_{n}) \subset \gcgb$ converges to
  $a\in \csrgb$.  Then $(\ju(f_{n}))$ converges to $\ju(a)$ in $\Xu$.
\end{cor}

\section{The Module $\Wg$}
\label{sec:module-w}

Fix $\gamma \in G$ and let $r(\gamma)=u$ and $s(\gamma)=v$.  Let
$\Wgo$ be the vector space of finitely supported functions
$\xi \colon G_{u}\to \B$ such that $\xi(\eta)\in B_{\eta\gamma}$.
Then $\Wgo$ carries an obvious right $A_{v}$-action:
$\xi\cdot a(\eta)=\xi(\eta)a$.  If $\xi,\zeta\in \Wgo$, then for each
$\eta\in G_{u}$, $\xi(\eta)^{*}\zeta(\eta)$ is just the $A_{v}$-valued
inner product in $B_{\eta\gamma}$.  Hence
\begin{equation}
  \label{eq:52}
  \rip A_{v}<\xi,\zeta>=\sum_{\eta\in G_{u}} \xi(\eta)^{*} \, \zeta(\eta)
\end{equation}
is easily seen to be an $A_{v}$-valued pre-inner product on $\Wgo$ as
in \cite{rw:morita}*{Lemma~2.16}.  We denote the right Hilbert module
completion by $\Wg$.

\begin{remark}
  \label{rem-ilt-dense} Since both $\xi$ and $\zeta$ have finite
  support, the norm of \eqref{eq:52} is bounded by a multiple of
  $\|\xi\|_{\infty}\|\zeta\|_{\infty}$.  Thus, if $(\xi_{i})$
  converges uniformly to $\xi$ with supports all contained in a fixed
  finite set, then $\xi_{i}\to \xi$ in the norm induced by the inner
  product.  In particular, since $p \colon \B\to G$ is saturated,
  sections of the form $\eta\mapsto f(\eta)b$ with $f\in \gcgub$ and
  $b\in B_{\gamma}$ span a dense subspace of $\Wgo$.
\end{remark}

Since $\Xu$ is a full right Hilbert $A_{u}$-module, it is also a
$\Ku\sme A_{u}$-\ib\ where $\Ku$ is the ideal in $\L(\Xu)$ generated
by the ``rank-one'' operators $\lip\Ku<\vec x,\vec y>$, defined for
all $\vec x,\vec y\in \Xu$ by
$\lip\Ku<\vec x,\vec y>(\vec z) :=\vec x\cdot \Ripu<\vec y,\vec
z>$. We let
\begin{equation}
  \label{eq:5}
  \ko=\operatorname{span}\set{\lip\Ku<f,g>:f,g\in\gcgub},
\end{equation}
a dense subalgebra of $\Ku$.

Note that, if $f,g,h\in\gcgub$, then $\lip\Ku<f,g>(h) \in\gcgub$, and
\begin{equation}
  \label{eq:53}
  \lip\Ku<f,g>(h)(\eta)=f(\eta)\cdot \Ripu<g,h>\quad\text{for all
    $\eta\in G_{u}$,}
\end{equation}
where the product on the right-hand side is the $A_{u}$-action on the
fibre $B_{\eta}$.  It follows that, for $\gamma\in G_v^u$, the
interior tensor product $\Xu\tensor_{A_{u}}B_{\gamma}$ is a
$\Ku\sme A_{v}$-\ib\ when equipped with the inner products from
\cite{rw:morita}*{Proposition~3.16}.

\begin{lemma}\label{lem-psi-iso}
  The map $\Psi$ determined by
  $f\tensor b\mapsto (\eta\mapsto f(\eta)b)$ from
  $\gcgub\atensor B_{\gamma}$ to $\Wgo$ extends to a right Hilbert
  $A_{v}$-module isomorphism of $\Xu\tensor_{A_{u}} B_{\gamma}$ onto
  $\Wg$.
\end{lemma}
\begin{proof}
  The map $\Psi$ is clearly $A_{u}$-balanced and bilinear.  Hence
  $\Psi$ indeed defines a map from $\gcgub\atensor B_{\gamma}$ to
  $\Wgo$.  Since
  \begin{align}
    \rip A_{v}<f\tensor b, g\tensor c>
    &= \brip A_{v}<{\Ripu<g,f>}\cdot b,c> \\
    &= b^{*}\Ripu<f,g>c \\
    &=\sum_{\eta\in G_{u}}b^{*} \, 
      f(\eta)^{*} \, g(\eta) \, 
      c  \\
    &= \rip A_{v}<\Psi(f\tensor b),\Psi(g\tensor c)>,
  \end{align}
  it follows that $\Psi$ preserves the inner product.  It follows from
  Remark~\ref{rem-ilt-dense} that $\Psi$ has dense range and the
  result follows.
\end{proof}

Using Lemma~\ref{lem-psi-iso}, we obtain the following.

\begin{prop}
  \label{prop-wg-ib} We can view $\Wg$ as a $\Ku\sme A_{v}$-\ib\ with
  respect to the left $\Ku$-action determined by
  \begin{equation}
    \label{eq:7}
    \bigl[\lip\Ku<f,g>\cdot \xi\bigr]
    (\eta)=f(\eta)\sum_{\tau\in
      G_{u}}g(\tau)^{*} \, \xi(\tau) \quad\text{for $f,g\in\gcgub$ and
      $\xi\in \Wgo$,}
  \end{equation}
  and the left $\Ku$-valued inner product determined by
  \begin{equation}
    \label{eq:8}
    \bigl[\lip\Ku<\xi,\zeta>\cdot \omega\bigr]
    (\eta)=\xi(\eta) \sum_{\tau\in G_{u}}
    \zeta(\tau)^{*}  \, \omega(\tau)\quad\text{for $\xi,\zeta,\omega\in\Wgo$.}
  \end{equation}
\end{prop}
\begin{proof}
  [Sketch of the Proof] Since we can identify
  $\K(\Xu\tensor_{A_{u}}B_{\gamma})$ with $\Ku$, and since by
  \cite{rw:morita}*{Lemma~4.55}, $T\mapsto \Psi T \Psi^{-1}$ is an
  isomorphism of $\Ku$ with $\K(\W_{\gamma})$, it suffices to see that
  the induced action of $\Ku$ and the left inner product are as
  specified on $\ko$.

  Suppose that $\xi=\Psi(h\tensor b)$.  Then
  \begin{align}
    \Psi\bigl(\lip\Ku<f,g>(h)\tensor b\bigr)(\eta)
    &= \lip\Ku<f,g>(h)(\eta)b \\
    &= f(\eta)\cdot \Ripu<g,h> b \\
    &= f(\eta) \sum_{\tau\in G_{u}} g(\tau)^{*} \, 
      \xi(\tau).
  \end{align}
  Hence for $\xi$ of the form $\eta\mapsto h(\eta)b,$ the formula for
  $\lip\Ku<f,g>\cdot \xi$ in \eqref{eq:7} indeed coincides with
  $\Psi \bigl(\lip\Ku<f,g>\cdot \Psi^{-1}(\xi)\bigr)$.  Since such
  elements $\xi$ span a dense subspace, we have established
  \eqref{eq:7}.

  By \cite{rw:morita}*{Lemma~4.55}, the $\K(\Wg)$-valued inner product
  on $\Wg$ is given by
  \begin{align}
    \blip \Wg <\Psi(f\tensor b),\Psi(g\tensor c)>= \Psi\circ \lip\Ku<f\tensor
    b,g\tensor c>)\circ \Psi^{-1}.
  \end{align}

  On the other hand,
  \begin{equation}
    \label{eq:9}
    \lip\Ku<f\tensor b,g\tensor c>=\blip\Ku<f,g\cdot {\lip A_{u}<c,b>}> =
    \lip \Ku<f,g \cdot cb^{*}>.
  \end{equation}
  It follows that
  \begin{equation}
    \label{eq:19}
    \lip\Ku<f\tensor b,g\tensor c>(h\tensor d)= \lip \Ku<f,g \cdot
    cb^{*}>(h) \tensor d.
  \end{equation}
  If we let $\xi=\Psi(f\tensor b)$, $\zeta=\Psi(g\tensor c)$, and
  $\omega=\Psi (h\tensor d)$, then
  \begin{align}
    \Psi(\lip\Ku<f\tensor b,g\tensor c>
    \cdot
    (h\tensor d))(\eta)
    &= \lip\Ku<f,g \cdot
      cb^{*}>(h)(\eta) \, d \\
    &= f(\eta)\sum_{\tau\in G_{u}}bc^{*} \, g(\tau)^{*} \, h(\tau) \, d \\
    &= \xi(\eta)\sum_{\tau\in G_{u}}
      \zeta(\tau)^{*} \, \omega(\tau).
  \end{align}
  Hence \eqref{eq:8} holds by continuity.
\end{proof}

Just as in Lemma~\ref{lem-xu-form}, we can realize $\Wg$ as follows.
\begin{lemma}
  \label{lem-wg-form} If $\gamma\in G_{v}^{u}$, then we can realize
  $\Wg$ with the right Hilbert $A_{v}$-module
  \begin{equation}
    \label{eq:20}
    \Bigl\{\,\vec x \in \prod_{\eta\in G_{u}}B_{\eta\gamma}:
    \text{$\sum_{\eta\in G_{u}}\vec x(\eta)^{*} \, \vec x(\eta)$ converges
      in $A_{v}$}\,\Bigr\}
  \end{equation}
  equipped with the inner product and right $A_{v}$-action given by
  \begin{equation}
    \label{eq:21}
    \rip A_{v}<\vec x,\vec y>=\sum_{\eta\in G_{u}} \, \vec x(\eta)^{*}\vec
    y(\eta) \quad\text{and} \quad (\vec x\cdot a) (\eta)
    =\vec x (\eta)\cdot a.
  \end{equation}

\end{lemma}

\section{The Convolution Formula}
\label{sec:convolution-formula}

\begin{prop}
  \label{prop-key-iso} Suppose that $\gamma\in G_{v}^{u}$ and that
  $\dX_{u}$ is the $A_{u}\sme \Ku$-\ib\ dual to $\Xu$.  Then there is
  a right Hilbert $A_{v}$-module isomorphism
  $\Phi \colon \dX_{u}\tensor_{\Ku}\Wg\to B_{\gamma}$ given on
  elementary tensors
  $\dual f\tensor \xi\in \dual{\gcgub}\atensor \Wgo$ by
  \begin{equation}
    \label{eq:22}
    \Phi(\dual f\tensor \xi)=\sum_{\eta\in G_{u}}f(\eta)^{*} \, \xi(\eta).
  \end{equation}
\end{prop}
\begin{proof}
  It is straightforward to see that $\Phi$ is $\ko$-balanced,
  bilinear, and $A_{v}$-linear.

  On the other hand,
  \begin{align}
    \brip A_{v}<\dual f\tensor \xi,\dual g \tensor \zeta>
    &= \rip A_{v}< {\lip \Ku<g,f>}\cdot \xi,\zeta> \\
    &= \sum_{\eta\in G_{u}} \lip\Ku<g,f>(\xi)(\eta)^{*} \, \zeta(\eta) \\
    &= \sum_{\eta\in G_{u}} \sum_{\tau\in G_{u}}
      \bigl(g(\eta) \, f(\tau)^{*} \, \xi(\tau)\bigr)^{*} \, \zeta(\eta) \\
    &= \sum_{\eta\in G_{u}} \sum_{\tau\in G_{u}}
      (f(\tau)^{*} \, \xi(\tau))^{*}
      g(\eta)^{*} \, \zeta(\eta) \\
    &= \Phi(\dual f\tensor \xi)^{*} \, \Phi(\dual g\tensor \zeta) \\
    &= \brip A_{v}<\Phi(\dual f\tensor \xi),\Phi(\dual g\tensor \zeta)>.
  \end{align}
  It follows that $\Phi$ preserves the right $A_{v}$-pre-inner
  products and hence extends to the completion
  $\dX_{u}\tensor_{\Ku}\Wg$.  Since it clearly has dense range, this
  completes the proof.
\end{proof}

\begin{remark}
  \label{rem-alt-proof} We could also obtain
  Proposition~\ref{prop-key-iso} by observing that
  \begin{align}
    \label{eq:49}
    B_{\gamma}
    &\cong
      A_{u}\tensor_{A_{u}}B_{\gamma}\cong
      (\dX_{u}\tensor_{\Ku}\Xu)\tensor_{A_{u}} B_{\gamma}
      \cong
      \dX_{u}\tensor_{\Ku}(\Xu\tensor_{A_{u}}B_{\gamma})\\
    & \cong
      \dX_u
      \tensor_{\Ku}\Wg,
  \end{align}
  but a direct proof will be useful in the sequel.
\end{remark}

\begin{cor}
  \label{cor-sum-conv} After realizing $\Xu$ as in
  Lemma~\ref{lem-xu-form} and $\Wg$ as in Lemma~\ref{lem-wg-form}, the
  formula in \eqref{eq:22} extends to elementary tensors in
  $\dX_{u}\atensor \Wg$.  That is, if $\vec x \in \Xu$ and
  $\vec y\in \Wg$, then
  \begin{equation}
    \label{eq:23}
    \Phi(\dual{\vec x}\tensor \vec y)=\sum_{\eta\in G_{u}} \vec
    x(\eta)^{*} \,  \vec y(\eta).
  \end{equation}
  In particular, the sum in \eqref{eq:23} converges in $B_{\gamma}$,
  and
  \begin{equation}
    \label{eq:24}
    \Bigl\|\sum_{\eta\in G_{u}} \vec x(\eta)^{*} \,  \vec y
    (\eta)\Bigr\| \le \|\vec 
    x\|_{\Xu}\| \vec y\|_{\Wg}.
  \end{equation}
\end{cor}

\begin{proof}
  If $F$ is a finite subset of $G_{u}$, then let
  $\vec x_{F}=\charfcn F \vec x$ and $\vec y_{F}=\charfcn F \vec
  y$. Then, just as in the proof of Lemma~\ref{lem-xu-form}, the nets
  $(\vec x_{F})$ and $(\vec y_{F})$ converge to $\vec x$ in $\Xu$ and
  $\vec y$ in $\Wg$, respectively.  By Lemma~\ref{lem-cross-norm},
  the net $\bigl(\dual{\vec x_{F}} \tensor \vec y_{F} \bigr)$
  thus converges to $\dual {\vec x} \tensor \vec y$ in
  $\dX_{u}\tensor_{\Ku}\Wg$. Therefore
  $\bigl(\Phi(\dual{\vec x_{F}} \tensor \vec y_{F})\bigr)$ converges
  to $\Phi(\dual {\vec x} \tensor \vec y)$ in $B_{\gamma}$.  Since
  \begin{equation}
    \label{eq:33}
    \Phi(\dual{\vec
      x_{F}} \tensor \vec y_{F}) =\sum_{\eta\in F}\vec x(\eta)^{*} \,
    \vec y(\eta), 
  \end{equation}
  and since $B_{\gamma}$ is Hausdorff, this establishes \eqref{eq:23}.

  Since $\Phi$ preserves the inner product,
  \begin{equation}
    \label{eq:25}
    \Bigl\|\sum_{\eta\in F} \vec x(\eta)^{*}  \, \vec y (\eta)\Bigr\|
    =
    \|\Phi(\dual{\vec x_{F}}\tensor \vec y_{F})\|
    =\|\dual{\vec x_{F}}\tensor \vec y_{F}\|.
  \end{equation}
  Using Lemma~\ref{lem-cross-norm}, we conclude that
  \begin{equation}
    \Bigl\|\sum_{\eta\in F} \vec x(\eta)^{*} \,  \vec y (\eta)\Bigr\|
    \le \|\vec x_{F}\|_{\Xu} \|\vec y_{F}\|_{\Wg} \le \|\vec x\|_{\Xu}
    \|\vec y\|_{\Wg},
  \end{equation}
  and we thus obtain \eqref{eq:24} as well.
\end{proof}

\begin{cor}
  \label{cor-converges} Suppose that $a,b\in\csrgb$ and that
  $\gamma\in G^{u}_{v}$.  Then
  \begin{equation}
    \label{eq:30}
    \sum_{\eta\in G^{u}}j(a)(\eta) \, j(b)(\eta^{-1}\gamma)
  \end{equation}
  converges in $B_{\gamma}$.  Moreover,
  \begin{equation}
    \label{eq:36}
    j(a*b)(\gamma)=\sum_{\eta\in G^{u}}j(a)(\eta) \, j(b)(\eta^{-1}\gamma).
  \end{equation}
\end{cor}
\begin{proof}
  To see that \eqref{eq:30} converges, it suffices to see that
  \begin{equation}
    \label{eq:31}
    \sum_{\eta\in G_{u}} j(a)(\eta^{-1}) \, j(b)(\eta\gamma)
  \end{equation}
  converges.  Using Corollary~\ref{cor-j-star}, the sum in
  \eqref{eq:31} is the same as
  \begin{equation}
    \label{eq:32}
    \sum_{\eta\in G_{u}}j(a^{*})(\eta)^{*} \, j(b)(\eta\gamma).
  \end{equation}

  Now let $\vec x_{a}(\eta)=j(a^{*})(\eta)$ for $\eta\in G_{u}$.  Then
  by Proposition~\ref{prop-ja-xu}, $\vec x_{a}\in \Xu$ with
  $\|\vec x_{a}\|_{\Xu}\le \|a\|_{r}$.  Note that the same proposition
  states that
  \begin{align}
    \Bigl\|\sum_{\eta\in
    G_{u}}j(b)(\eta\gamma)^{*} \, j(b)(\eta\gamma) \Bigr\|
    =
    \Bigl\| \sum_{\eta\in G_{v}} j(b)(\eta)^{*} \, j(b)(\eta) \Bigr\| \le
    \|b\|_{r}^{2}. 
  \end{align}
  In particular, if we let $ \vec{y}_{b} (\eta)=j(b)(\eta\gamma)$ for
  $\eta\in G_{u}$, then $\vec y_{b}$ is an element of $\Wg$ with norm
  bounded by $\|b\|_{r}$.  Now Corollary~\ref{cor-sum-conv}, applied
  to $\dual{\vec x_{a}}\tensor \vec y_{b}$, immediately implies that
  \eqref{eq:30} converges.  Let $(f_{i})$ and $(g_{i})$ be sequences
  in $\gcgb$ converging in $\csrgb$ to $a$ and $b$, respectively.
  Then by the above
  \begin{align}
    \|\vec x_{f_{i}}-\vec x_{a}\|_{\Xu}=\|\vec x_{(f_{i}-a)}\|_{\Xu}\le
    \|f_{i}-a\|_{r}. 
  \end{align}
  Hence $(\vec x_{f_{i}})$ converges to $\vec x_{a}$ in $\Xu$.
  Similarly, $(\vec y_{g_{i}})$ converges to $\vec y_{b}$ in $\Wg$.
  Thus
  $\Phi(\dual{\vec x_{f_{i}}}\tensor \vec y_{g_{i}}) \to \Phi(\dual{
    \vec x_{a} }\tensor \vec y_{b})$.  But
  \begin{align}
    \Phi(\dual{\vec x_{f_{i}}}\tensor \vec
    y_{g_{i}})
    =(f_{i}*g_{i})(\gamma)=j(f_{i}*g_{i})(\gamma). 
  \end{align}
  This suffices since the continuity of $j$ and multiplication
  implies that $j(f_{i}*g_{i})(\gamma)$ converges to $j(a*b)(\gamma)$.
\end{proof}

Our main result, Theorem~\ref{thm-main}, now follows from
Proposition~\ref{prop-j-inj}, Corollary~\ref{cor-j-star}, and
Corollary~\ref{cor-converges}.

\section{Examples}
\label{sec:examples}

As illustrated in \cite{muhwil:dm08}*{\S2}, Fell bundles and their
\cs-algebras subsume most examples of groupoid dynamical systems.
Hence our main theorem applies to many such examples when the groupoid
involved is \'etale.  We consider some such examples here, and the first
recovers \cite{bfpr:nyjm21}*{Proposition~2.8}.

\begin{example}[Twists]\label{ex-twists}
  Consider a twist $E$ over an \'etale groupoid $G$ as defined by
  Kumjian in \cite{kum:cjm86}.  To be precise,
  we have a central groupoid
  extension
  \begin{equation}
    \label{eq:50}
    \begin{tikzcd}
      \go\times \T\arrow[r,"\iota"] &E \arrow[r,"q",two heads]& G
    \end{tikzcd}
  \end{equation}
  where $\iota$ and $q$ are continuous groupoid homomorphisms such
  that $\iota$ is a homeomorphism onto its range, $q$ is an open
  surjection with kernel equal to the range of $\iota$, and
  \begin{equation}
    \label{eq:51}
    \iota(r(e),z) \, e=e \, \iota(s(e),z)\quad\text{for all $e\in E$
      and $z\in \T$.} 
  \end{equation}
  The associated \cs-algebra $\cs(G;E)$ can be realized as the
  \cs-algebra of the Fell bundle $p \colon \B\to G$ where $\B$ is the
  quotient of $E\times\C$ by the $\T$-action
  $z\cdot (e,\lambda)=(ze,z\lambda)$ and $p$ is given by
  $p([e,z])=q(e) $.  Then $[e,\lambda][f,\tau]=[ef,\lambda\tau]$
  whenever $(e,f)\in E^{(2)}$, and
  $[e,\lambda]^{*}=[e^{-1},\overline \lambda]$.

  As in \cite{ikrsw:jfa20}*{Proposition~1.4}, it is not hard to see
  that if $\check f \colon G\to \B$ is a continuous section, then
  $\check f(q(e))=[e,f(e)]$ where $f \colon E\to \C$ is a continuous
  function such that\footnote{Unfortunately, the literature is
    inconsistent as to whether there should be a complex conjugate on
    the section $z$ in \eqref{eq:54} when constructing the \cs-algebra
    $\cs(G;E)$.  As explained in \cite{erpwil:jot14}*{Example~2.3},
    the choice depends on whether one takes the Fell bundle $\B$ to be
    the complex line bundle defined above or its conjugate bundle
    $\overline E$.  We have opted to stay consistent with Kumjian's
    choice in \cite{kum:cjm86}. Note that $\cs(G;E)$ and
      $\cs(G;\overline E)$ each others opposite algebras---see
      \cite{bussim:ijm21}.}
  \begin{equation}
    \label{eq:54}
    f(z\cdot e)=zf(e).
  \end{equation}
  Therefore, if we let $C_{0}(G;E)$ be the set of continuous functions
  $f \colon E\to\C$ satisfying \eqref{eq:54} and such that
  $q(e)\mapsto |f(e)|$ vanishes at infinity on $G$, then
  Theorem~\ref{thm-main} implies that there is a norm reducing
  injective linear map $j \colon \cs_{r}(G;E)\to C_{0}(G;E)$ such that
  if $\check f\in \gcgb$, then $j(\check f)=f$.  Moreover, if
  $a,b\in\csrgb$ then
  \begin{equation}
    \label{eq:55}
    j(a^{*})(e)=\overline{j(a)(e^{-1})} \quad\text{and} \quad
    j(a*b)(e')=\sum_{q(e)\in G^{r(e')}}j(a)(e) \, j(b)(e^{-1}e').
  \end{equation}
\end{example}

\begin{remark}\label{rem-sum-over-g}
  It should be kept in mind that the sum in \eqref{eq:55} is
    taken over the set $G^{r(e')}$.  The quantity $j(a)(e) \, j(b)(e^{-1}e')$
depends only on $q(e)$.   Hence if we define $F(q(e))=j(a)(e) \,
j(b)(e^{-1}e')$, then the sum is equal to
\begin{equation}
  \label{eq:38}
  \sum_{\eta\in G^{r(e')}}F(\eta).
\end{equation}
Furthermore, since the sum in \eqref{eq:55} (or \eqref{eq:38})
is invariant under
rearrangement, the convergence is absolute.
  \end{remark}

  \begin{example}[Groupoid Crossed Products] \label{ex-cross-prod} Let
    $(\E,G,\alpha)$ be a groupoid dynamical system with $G$ \'etale.
    To be precise, $ k \colon \E \to\go$ is a \cs-bundle and
    $\alpha=\set{\alpha_{\gamma}}_{\gamma\in G}$ is a family of
    isomorphisms
    $\alpha_{\gamma} \colon E_{s(\gamma)}\to E_{r(\gamma)}$ such that
    $\gamma\cdot a=\alpha_{\gamma}(a)$ is a continuous action of $G$
    on the bundle $\E$.  Then we can realize the reduced crossed
    product $\E\rtimes_{\alpha, r}G$ as the reduced \cs-algebra of a
    Fell bundle $p \colon \B\to G$ where
    $\B =r^{*}\E=\set{(a,\gamma)\in\E\times G:q(a)=r(\gamma)}$ with
    $p(a,\gamma)=\gamma$.  Then
    $(a,\gamma)(b,\eta)=(a\alpha_{\gamma}(b),\gamma\eta)$ for
      $(\gamma,\eta)\in G^{(2)}$, and
    $(a,\gamma)^{*} =(\alpha_{\gamma}^{-1}(a^{*}),\gamma^{-1})$.

  As above, if $\check f\in \Gamma_{0}(G;\B)$, then there is a
  continuous function $f \colon G\to \E$ such that
  $f(\gamma)\in E_{r(\gamma)}$ and $\gamma\mapsto \|f(\gamma)\|$
  vanishes at infinity on $G$.  If we denote the collection of such
  functions by $C_{0,r}(G,\E)$, then Theorem~\ref{thm-main} implies
  that there is a norm reducing injective linear map
  $j \colon \E\rtimes_{\alpha,r}G\to C_{0,r}(G,\E)$ such that
  $j(\check f)=f$ for all $f\in \Gamma_{c}(G;\B)$.  Moreover, if
  $a,b\in \E\rtimes_{\alpha,r}G$, then
  \begin{equation}
    \label{eq:56}
    j(a^{*})(\gamma)=j(a)(\gamma^{-1})^{*} \quad \text{and} \quad
    j(a*b)(\gamma) = \sum_{\eta\in G^{r(\gamma)}} j(a)(\eta) \, 
    \alpha_{\eta}\bigl(j(b)(\eta^{-1}\gamma)\bigr) .
  \end{equation}
\end{example}

  \begin{remark}
    \label{rem-group-case} Of course, Example~\ref{ex-cross-prod} 
    applies to crossed products by discrete groups.  In that case, the
    result has been known for some time and appears in Zeller-Meier
    \cite{zel:jmpa68}*{Theorem~4.2} with a weaker notion of
    convergence for the convolution product.  Zeller-Meier also
    allows for a unitary valued $2$-cocycle.  However, using the
    observations in \cite{exelac:pams97}, we can use a Fell bundle
    model to include cocycles.  We omit the details.
  \end{remark}

\begin{example}[Green-Renault Twisted Crossed Products]
  We can combine the idea of a twist and a groupoid crossed product to
  arrive at Renault's generalization of a Green twisted crossed
  product from \cite{ren:jot91}.  As in
  \cite{muhwil:dm08}*{Example~2.5} or the slightly more general set-up
  in \cite{ikrsw:jfa20}*{\S1.4}, we can realize these twisted crossed
  products via a Fell bundle.  We start with a 
  groupoid extension that fixes the unit space,
  \begin{equation}
    \label{eq:59}
    \begin{tikzcd}
      \A\arrow[r,"\iota",hook] & \Sigma\arrow[r,"q",two heads] & G
    \end{tikzcd}
  \end{equation}
  where $\A$ is a subgroupoid group bundle of $\Sigma$ with unit space
  $\go$, $\iota$ is the inclusion map, and $q$ is a continuous open
  surjection restricting to a homeomorphism of $\Sigma^{(0)}$ with
  $\go$.  (Hereafter, 
  we identify $\Sigma^{(0)}$ with $\go$.)  We also
  require a groupoid dynamical system $(\E
  ,\Sigma,\alpha)$ for a \cs-bundle $k\colon \E\to\go$ as in
  Example~\ref{ex-cross-prod}.

  To define the twist, we let $U(E_{u})$ for $u\in \go$ be the unitary
  group of the \cs-algebra $E_{u}$ and let $\coprod_{u\in\go}U(E_{u})$
  be the corresponding (algebraic) group bundle over $\go$.  Then a
  {\em twisting map} is a unit-space fixing homomorphism
  $\theta \colon \A\to \coprod_{u\in\go}U(E_{u})$ that induces an
  action by isometric Banach space isomorphisms of $\A$ on $\E$ such
  that $(a,e)\mapsto a\cdot e \coloneqq \theta(a)e$ is continuous
  from $\A*\E = \set{ (a, e)\in \A\times \E: s(a)=k(e) }$ to
    $\E$, and which satisfies
  \begin{align}
    \label{eq:48}
    \alpha_{a}(e)
    &=\theta(a)e\theta(a)^{*} &&\text{for all
                                 $(a,e)\in \A*\E$, and} \\
    \theta(\sigma a\sigma^{-1})
    &=
      \overline{\alpha}_{\sigma}(\theta(a))&& \text{for all
                                              $(\sigma,a)\in \Sigma*\A$.}
  \end{align}
  
  To get a Fell bundle, we observe that $\A$ acts on $r^{*}\E$ by
  $a\cdot (\sigma,e)=(a\sigma,e\theta(a)^{*})$.  Then the left
  quotient $\B=\A\backslash r^{*}\E$ is a Banach bundle over $G$ with
  $p \colon \B\to G$ given by $p([\sigma,e])= q(\sigma) $.  Then $\B$
  is a Fell bundle with
  \begin{equation}
    \label{eq:68}
    [\sigma,e][\tau,f]=[\sigma\tau,e\alpha_{\sigma}(f)]\quad\text{and}
    \quad [\sigma,e]^{*}=[\sigma^{-1},\alpha_{\sigma}^{-1}(e^{*})].
  \end{equation}
  Using \cite{ikrsw:jfa20}*{Proposition~1.4},
  we see that sections
  $\check f\in \Gamma(G;\B)$ correspond to continuous functions
  $f \colon \Sigma \to\E$ such that
  $f(a\sigma)=f(\sigma)\theta(a)^{*}$ for all $(a,\sigma)\in \A*\E$.
  If we let
  $C_{0}(G;\Sigma;\E)$
  be the continuous functions
  $f \colon \Sigma\to \E$ transforming as above and such that
  $\gamma\mapsto \|\check{f}(\gamma)\|$
  vanishes at infinity on $G$, then
  Theorem~\ref{thm-main} implies that there is an injective norm
  reducing linear map $j \colon \csrgb\to C_{0}(G;\Sigma;\E)$
  such that
  for all $a,b\in \csrgb$ we have
  \begin{equation}
    \label{eq:69}
    j(a^{*})(\sigma)=j(a)(\sigma^{-1})^{*} \quad \text{and} \quad
    j(a*b)(\sigma) = \sum_{q(\tau)\in G^{r(\sigma)}}j(a)(\tau) \,
      \alpha_{\tau}\bigl(j(b)(\tau^{-1}\sigma)\bigr)
  \end{equation}
where the comments in Remark~\ref{rem-sum-over-g} apply to the
  sum since $j(a)(\tau) \,
  \alpha_{\tau}\bigl(j(b)(\tau^{-1}\sigma)\bigr)$ depends only on $q(\tau)$.
\end{example}



\def\noopsort#1{}\def\cprime{$'$} \def\sp{^}

\end{document}